\documentclass[12pt]{article}

\usepackage{amsmath,amssymb,enumerate,bbm}
\newenvironment{proof}{\noindent\textbf{Proof\ }}{\hspace*{\fill}$\Box$\medskip}
\newtheorem{theorem}{Theorem}

\newtheorem{lemma}{Lemma}
\newtheorem{corollary}{Corollary}

\begin{document}

\title{Szemer\'edi-Trotter type theorem and sum-product estimate in finite
fields}\author{Le Anh Vinh\\
Mathematics Department\\
Harvard University\\
Cambridge, MA 02138, US\\
vinh@math.harvard.edu}\maketitle

\begin{abstract}
  We study a Szemer\'edi-Trotter type theorem in finite fields. We then use
  this theorem to obtain an improved sum-product estimate in finite fields.
\end{abstract}

\section{Introduction}

Let $A$ be a non-empty subset of a finite field $F_q$. We consider the sum set
\[ A+A:=\{a+b : a, b \in A\}\]
and the product set
\[A.A:=\{a.b : a, b \in A\}.\]

Let $|A|$ denote the cardinality of $A$. Bourgain, Katz and Tao (\cite{bourgain-katz-tao}) showed, using an argument of Edgar and Miller \cite{edgar-miller}, that when $ 1 \ll |A| \ll q$ then
\[\max(|A+A|,|A.A|) \gg |A|;\]
this improves the easy bound $|A+A|, |A.A| \geqslant |A|$. The precise statement of the sum-product estimate is as follows.

\begin{theorem}(\cite{bourgain-katz-tao}) \label{t1}
Let $A$ be a subset of $F_q$ such that
\[ q^\delta < |A| < q^{1-\delta}\]
for some $\delta > 0$. Then one has a bound of the form
\[ \max (|A+A|, |A.A|) \geqslant c(\delta)|A|^{1+\epsilon}\]
for some $\epsilon = \epsilon(\delta) > 0$.
\end{theorem}

Using Theorem \ref{t1}, Bourgain, Katz and Tao can prove a theorem of Szemer\'edi-Trotter type in two-dimensional finite field geometries. 
Roughly speaking, this theorem asserts that if we are in the finite plane $F_q^2$ and one has $N$ lines and $N$ points in that plane for some $1 \ll N \ll q^2$,
then there are at most $O(N^{3/2-\epsilon})$ incidences; this improves the standard bound of $O(N^{3/2})$ obtained from extremal graph theory. The precise statement of the theorem is as follows. 

\begin{theorem} (\cite{bourgain-katz-tao} \label{t2}
Let $P$ be a collection of points and $L$ be a collection of
lines in $F^2$. For any $0 < \alpha < 2$, if $|P|, |L| \leqslant N = q^{\alpha}$ then we have
\[ |\{(p, l) \in P \times L : p \in l\}| \leqslant C N^{3 / 2 - \varepsilon},
\]
for some $\varepsilon = \varepsilon (\alpha) > 0$ depending only on the
exponent $\alpha$. 
\end{theorem}

The relationship between $\varepsilon$ and $\delta$ in Theorem \ref{t1} and the relationship between $\alpha$ and $\varepsilon$ in Theorem \ref{t2}
however are difficult to determine. 

In this paper we shall proceed in an opposite direction. We will first prove a theorem of Szemer\'edi-Trotter type about the number of incidences between points and lines in finite field geometries. We then apply this result to obtain an improved sum-product estimate. Our first result is the following.

\begin{theorem}\label{m1}
  Let $P$ be a collection of points and $L$ be a collection of lines in
  $F^2_q$. Then we have
  \begin{equation}
    |\{(p, l) \in P \times L : p \in l\}| \leqslant \frac{|P\|L|}{q} +
    q^{1/2}\sqrt{|P\|L|} .
  \end{equation}
\end{theorem}

In the spirit of Bourgain-Katz-Tao's result, we obtain a reasonably good
estimate when $1 < \alpha < 2$.

\begin{corollary} \label{cor1}
  Let $P$ be a collection of points and $L$ be a collection of lines in
  $F^2_q$. Suppose that $|P|, |L| \leqslant N = q^{\alpha}$ with $1 +
  \varepsilon \leqslant \alpha \leqslant 2 - \varepsilon$ for some
  $\varepsilon > 0$. Then we have
  \begin{equation}
    |\{(p, l) \in P \times L : p \in l\}| \leqslant 2 N^{\frac{3}{2} -
    \frac{\varepsilon}{4}} .
  \end{equation}
\end{corollary}

We shall use the incidence bound in Theorem \ref{m1} to obtain an improved sum-product
estimate. 

\begin{theorem}\label{m2} (sum-product estimate)
  Let $A \subset F_q$ with $q$ is an odd prime power. Suppose that
  \[ |A + A| = m, |A.A| = n. \]
  Then
  \[ |A|^2 \leqslant \frac{m n|A|}{q} + q^{1/2}\sqrt{m n} . \]
  In particular, we have
  \begin{equation}
    \max (|A + A|, |A.A|) \geqslant \frac{2| A|^2}{q^{1 / 2} + \sqrt{q +
    \frac{4| A|^3}{q}}} .
  \end{equation}
\end{theorem}

In analogy with the statement of Corollary \ref{cor1} above, we note the following consequence of Theorem \ref{m2}.

\begin{corollary} \label{cor2}
Let $A \subset F_q$ with $q$ is an odd prime power. 
\begin{enumerate}
	\item Suppose that $q^{1/2} \ll |A| \leqslant q^{2/3}$. Then
	\[ \max (|A + A|, |A.A|) \geqslant c\frac{|A|^2}{q^{1/2}}.\]
	\item Suppose that $q^{2/3} \leqslant |A| \ll q$. Then 
	\[ \max (|A + A|, |A.A|) \geqslant c(q|A|)^{1/2}. \]
\end{enumerate}
\end{corollary}

Note that, the bound in Theorem \ref{m2} is stronger than ones established in
Theorem 1.1 in \cite{hart-iosevich-solymosi}. 

We also call the reader's attention to the fact that the application of the spectral
method from graph theory in sum-product estimates was
independently used by Vu in \cite{van}. The bound in Corollary \ref{cor2} is stronger than ones in Remark 1.4 from \cite{van} (which is also implicit from Theorem 1.1 in \cite{hart-iosevich-solymosi}).

\section{Incidences: Proofs of Theorem \ref{m1} and Corollary \ref{cor1}}
  We can embed the space
  $F_q^2$ into $P F_q^3$ by identifying $(x,y)$ with the equivalence
  class of $(x,y, 1)$. Any line in $F_q^2$ also can be
  represented uniquely as an equivalence class in $P F_q^3$ of some non-zero element $h \in
  F_q^3$. For each $x \in F_q^3$, we denote $[x]$ the equivalence class of $x$ in $GF_q^3$. Let $G_q$ denote the graph whose vertices are the points of $P F_q^3$, where two vertices $[x]$ and $[y]$ are connected if and only if
  \[\left\langle x, y \right\rangle = x_1 y_1 + x_2 y_2 + x_3 y_3 =
  0.\] That is the points represented by $[x]$ and $[y]$ lie on the lines
  represented by $[y]$ and $[x]$, respectively.
  
  It is well-known that $G_q$ has $n = q^2+q+1$
  vertices and $G_q$ is a $(q+1)$-regular graph. Since the equation $x_1^2+x_2^2+x_3^2 = 0$ over $F_q$ has exactly $q^2-1$ non-zero solutions so the number of vertices of $G$
  with loops is $d = q+1$. The eigenvalues of $G$ are easy to compute. Let $A$ be the adjacency matrix of $G$. Since two lines in $PF_q^3$ intersect at exactly one point, we have
  $A^2 = A A^T = J+ (d-1)I = J + q I $ where $J$ is the $n \times n$ all $1$-s matrix and $I$ is the $n \times n$ identity matrix. Thus the largest eigenvalue of $A^2$ is $d^2$ and all other eigenvalues are $d-1 = q$. This implies that all but the largest eigenvalues of $G_q$ are $\sqrt{q}$.
  
  It is well-known that if a $k$-regular graph on $n$ vertices with the absolute value of each of its eigenvalues but the largest one is at most $\lambda$ and if $\lambda \ll d$ then this graph behaves similarly as a random graph $G_{n,k/n}$. Presicely, we have the following result (see Corollary~9.2.5 in \cite{alon-spencer}).
  
\begin{lemma} \label{expander}
    Let $G$ be a $k$-regular graph on $n$ vertices (with loops allowed).
    Suppose that all eigenvalues of $G$ except the largest one are at most
    $\lambda$. Then for every set of vertices $B$ and $C$ of $G$, we have
    \begin{equation} \label{eb}
      |e (B, C) - \frac{k}{n} |B||C\| \leqslant \lambda \sqrt{|B||C|},
    \end{equation}
    where $e (B, C)$ is the number of ordered pairs $(u, v)$ where $u \in B, v
    \in C$ and $u v$ is an edge of $G$.
  \end{lemma}
  
Let $B$ be the set of vertices of $G$ that represent the collection $P$ of
  points in $F_q^2$ and $C$ be the set of vertices of $G$ that represent the
  collection $L$ of lines in $F_q^2$. From (\ref{eb}), we have
  \begin{eqnarray*}
    |\{(p, h) \in P \times L : p \in h\}| & = & e (B, C)\\
    & \leqslant & \frac{q+1}{q^2+q+1} |B\|C| + \lambda \sqrt{|B\|C|}\\
    & \leqslant & \frac{|P\|L|}{q} + q^{1/ 2} \sqrt{|P\|L|}.
  \end{eqnarray*}
  This concludes the proof of Theorem \ref{m1}. 
  
If $\alpha \leqslant 2 - \varepsilon$ then
  \begin{equation} \label{a1}
    \frac{|P\|L|}{q} \leqslant \frac{N^2}{q} \leqslant N^{\frac{3}{2} -
    \frac{\varepsilon}{4}} .
  \end{equation}
  
If $\alpha \geqslant 1 + \varepsilon$ then
  \begin{equation} \label{a2}
    q^{1 / 2} \sqrt{|P\|L|} \leqslant q^{1 / 2} N \leqslant N^{\frac{3}{2} -
    \frac{\varepsilon}{4}} .
  \end{equation}
Corollary \ref{cor1} are immediate from (\ref{a1}), (\ref{a2}) and Theorem \ref{m1}.

Note that we also have an analog of Theorem \ref{m1} in higher dimension.

\begin{theorem}\label{m3}
  Let $P$ be a collection of points in $F_q^d$ and $H$ be a collection of
  hyperplanes in $F_q^d$ with $d \geqslant 2$. Then we have
  \begin{equation}
    |\{(p, h) \in P \times H : p \in h\}| \leqslant \frac{|P\|L|}{q} + q^{(d -
    1) / 2} (1+o(1)) \sqrt{|P\|L|} .
  \end{equation}
\end{theorem}

  The proof of this theorem is similar to the proof of Theorem \ref{m1} and is left for the readers. Note that the analog of Theorem 	  \ref{m3} for the case $P \equiv L$ (in $PF_q^d$) are obtained by Alon and Krivelevich (\cite{alon-krivelevich}) via a similar approach and by Hart, Iosevich, Koh and Rudnev (\cite{hart-iosevich-koh-rudnev}) via Fourier analysis. By modifying the proofs of Theorem 2.1 in \cite{hart-iosevich-koh-rudnev} and Lemma 2.2 in \cite{alon-krivelevich} slightly, we obtain Theorem \ref{m3}.

\section{Sum-product estimates: Proofs of Theorem \ref{m2} and Corollary \ref{cor2}}

Elekes (\cite{elekes}) observed that there is a connection between the incidence problem
and the sum-product problem. The statement and the proof here follow the presentation in \cite{bourgain-katz-tao}.

\begin{lemma} (\cite{elekes}) \label{incidence}
  Let $A$ be a subset of $F_q$. Then there is a collection of points $P$
  and lines $L$ with $|P| = |A + A||A.A|$ and $|L| = |A|^2$ which has at least
  $|A|^3$ incidences.
\end{lemma}

\begin{proof} Take $P = (A + A)\times (A.A)$, and let $L$ be the collection of all lines of form
$l(a,b) := \{(x,y): y = b(x-a)\}$ where $a, b \in A$. The claim follows since $(a+c,bc) \in P$ is incident to $l(a,b)$ whenever $a,b,c \in A$.

\end{proof}

Theorem \ref{m2} follows from Theorem \ref{m1} and Lemma \ref{incidence}.

\begin{proof}
  (of Theorem \ref{m2}) Let $P$ and $L$ be collections of points and lines as in the proof of
  Lemma \ref{expander}. Then from Theorem \ref{m1}, we have
  \[
    |A|^3 \leqslant \frac{m n|A|^2}{q} + q^{1 / 2} |A| \sqrt{m n} .
  \]
  This implies that
  \begin{equation}
    |A|^2 \leqslant \frac{m n|A|}{q} + q^{1 / 2} \sqrt{m n} .
  \end{equation}
  Let $x = \max (|A + A|, |A.A|)$, we have
  \[ |A|x^2 + q^{3 / 2} x - q|A|^2 \geqslant 0. \]
  Solving this inequality gives us the desired lower bound for $x$, concluding the proof
  of the theorem.
\end{proof}

If $q^{1/2} \ll |A| \leqslant q^{2/3}$. Then
	\begin{equation}\label{a3} q^{1/2} + \sqrt{q+\frac{4|A|^3}{q}} = O(q^{1/2}).\end{equation}
	
If $q^{2/3} \leqslant |A| \ll q$. Then 
	\begin{equation}\label{a4} q^{1/2} + \sqrt{q+\frac{4|A|^3}{q}} = O(\sqrt{|A|^3/q}). \end{equation}

Corollary \ref{cor2} is immediate from (\ref{a3}), (\ref{a4}) and  Theorem \ref{m2}.

\end{document}